\documentclass[reqno]{amsart}
\usepackage[english]{babel}
\usepackage{amssymb,amsmath}
\usepackage{bbold}
\usepackage{amsrefs}
\usepackage[foot]{amsaddr}
\swapnumbers

\def\intern{\textsf{\textup{intern}}}
 
\newcommand{\tup}{\underline} 

\newcommand{\Sh}{\ensuremath{\protect{S_{\st{}}}}}

\newcommand{\forallst}{\forall^{\st{}}}
\newcommand{\existsst}{\exists^{\st{}}}

\newtheorem{thm}{Theorem}

\newtheorem{cor}[thm]{Corollary}
\newtheorem{defi}[thm]{Definition}
\newtheorem{rem}[thm]{Remark}
\newtheorem{nota}[thm]{Notation}

\newcommand\be{\begin{equation}}
\newcommand\ee{\end{equation}}

\def\bdefi{\begin{defi}\rm}
\def\edefi{\end{defi}}
\def\bnota{\begin{nota}\rm}
\def\enota{\end{nota}}
\def\brem{\begin{rem}\rm}
\def\erem{\end{rem}}

\def\IST{\textup{\textsf{IST}}}

\def\H{\textup{\textsf{H}}}
\def\RCA{\textup{\textsf{RCA}}}

\def\RCAo{\textup{\textsf{RCA}}_{0}^{\omega}}

\def\G{\mathcal{G}}

\def\T{\mathcal{T}}

\def\bye{\end{document}}

\def\P{\textup{\textsf{P}}}

\def\R{{\mathbb  R}}

\def\R{{\mathbb{R}}}
\def\({\textup{(}}
\def\){\textup{)}}

\def\st{\textup{st}}

\def\al{\textup{\textsf{al}}}
\def\asa{\leftrightarrow}

\def\di{\rightarrow}

\def\ACA{\textup{\textsf{ACA}}}

\def\QFAC{\textup{\textsf{QF-AC}}}

\def\HAC{\textup{\textsf{HAC}}}

\def\INT{\textup{\textsf{int}}}

\numberwithin{equation}{section}
\numberwithin{thm}{section}

\begin{document}
\title{The computational content of the Loeb meausre}
\author{Sam Sanders}
\address{Munich Center for Mathematical Philosophy, LMU Munich, Germany}
\email{sasander@me.com}


\begin{abstract}
The \emph{Loeb measure} is one of the cornerstones of \emph{Nonstandard Analysis}.  
The traditional development of the Loeb measure makes use of \emph{saturation} and \emph{external sets}.  
Inspired by \cite{pimpson}, we give meaning to special cases of the Loeb measure in the weak fragment $\P$ of Nelson's \emph{internal set theory} from \cite{brie}.  
Perhaps surprisingly, our definition of the Loeb measure has computational content in the sense of the `term extraction' framework from \cite{brie}
\end{abstract}

\maketitle
%


\section{Introduction}\label{intro}
The \emph{Loeb measure} (\cite{loeb1,nsawork}) is one of the cornerstones of Robinson's \emph{Nonstandard Analysis} (NSA for short; see \cite{robinson1}).  
The traditional development of the Loeb measure in NSA makes use of \emph{saturation} and \emph{external sets}.  
A special case of the Loeb measure is introduced in a weak fragment of NSA in \cite{pimpson} using external sets, but without the use of \emph{saturation}.  
The definition of measure from Reverse Mathematics (RM for short see \cite{simpson2} for an overview) is used.  

\medskip

In this paper, we similarly introduce a special case of the Loeb measure, but in the weak fragment $\P$ of Nelson's \emph{internal set theory} ($\IST $ for short; see\cite{wownelly}) from \cite{brie}.  
We show that our definition of the Loeb measure has computational content in the sense of the framework from \cite{brie}.  In particular, we show that our definition of the Loeb measure falls inside the scope of the `term extraction theorem' of the system $\P$ as in Corollary~\ref{consresultcor} below.   

\medskip

We first introduce Nelson's internal set theory in Section \ref{IIST} and a fragment called $\P$ based on G\"odel's system \textsf{T} in Section \ref{PIPI}.  
The development of the Loeb measure in $\P$ takes place in Section \ref{LEB}.  


\section{Internal set theory and its fragment $\P$}\label{P}
In this section, we discuss Nelson's \emph{internal set theory}, first introduced in \cite{wownelly}, and its fragment $\P$ from \cite{brie}.  
The latter fragment is essential to our enterprise, especially Corollary~\ref{consresultcor} below.  
\subsection{Internal set theory 101}\label{IIST}
In Nelson's \emph{syntactic} approach to Nonstandard Analysis (\cite{wownelly}), as opposed to Robinson's semantic one (\cite{robinson1}), a new predicate `st($x$)', read as `$x$ is standard' is added to the language of \textsf{ZFC}, the usual foundation of mathematics.  
The notations $(\forall^{\st}x)$ and $(\exists^{\st}y)$ are short for $(\forall x)(\st(x)\di \dots)$ and $(\exists y)(\st(y)\wedge \dots)$.  A formula is called \emph{internal} if it does not involve `st', and \emph{external} otherwise.   
The three external axioms \emph{Idealisation}, \emph{Standard Part}, and \emph{Transfer} govern the new predicate `st';  They are respectively defined\footnote{The superscript `fin' in \textsf{(I)} means that $x$ is finite, i.e.\ its number of elements are bounded by a natural number.} as:  
\begin{enumerate}
\item[\textsf{(I)}] $(\forall^{\st~\textup{fin}}x)(\exists y)(\forall z\in x)\varphi(z,y)\di (\exists y)(\forall^{\st}x)\varphi(x,y)$, for internal $\varphi$ with any (possibly nonstandard) parameters.  
\item[\textsf{(S)}] $(\forall^{\st} x)(\exists^{\st}y)(\forall^{\st}z)\big((z\in x\wedge \varphi(z))\asa z\in y\big)$, for any $\varphi$.
\item[\textsf{(T)}] $(\forall^{\st}t)\big[(\forall^{\st}x)\varphi(x, t)\di (\forall x)\varphi(x, t)\big]$, where $\varphi(x,t)$ is internal, and only has free variables $t, x$.  
\end{enumerate}
The system \textsf{IST} is (the internal system) \textsf{ZFC} extended with the aforementioned external axioms;  
The former is a conservative extension of \textsf{ZFC} for the internal language, as proved in \cite{wownelly}.    

\medskip

In \cite{brie}, the authors study G\"odel's system $\textsf{T}$ extended with special cases of the external axioms of \textsf{IST}.  
In particular, they consider the systems $\H$ and $\P$, introduced in the next section, which are conservative extensions of the (internal) logical systems \textsf{E-HA}$^{\omega}$ and $\textsf{E-PA}^{\omega}$, respectively \emph{Heyting and Peano arithmetic in all finite types and the axiom of extensionality}.       
We refer to \cite{kohlenbach3}*{\S3.3} for the exact definitions of the (mainstream in mathematical logic) systems \textsf{E-HA}$^{\omega}$ and $\textsf{E-PA}^{\omega}$.  
Furthermore, \textsf{E-PA}$^{\omega*}$ and $\textsf{E-HA}^{\omega*}$ are the definitional extensions of \textsf{E-PA}$^{\omega}$ and $\textsf{E-HA}^{\omega}$ with types for finite sequences, as in \cite{brie}*{\S2}.  For the former systems, we require some notation.  
\begin{nota}[Finite sequences]\label{skim}\rm
The systems $\textsf{E-PA}^{\omega*}$ and $\textsf{E-HA}^{\omega*}$ have a dedicated type for `finite sequences of objects of type $\rho$', namely $\rho^{*}$.  Since the usual coding of pairs of numbers goes through in both, we shall not always distinguish between $0$ and $0^{*}$.  
Similarly, we do not always distinguish between `$s^{\rho}$' and `$\langle s^{\rho}\rangle$', where the former is `the object $s$ of type $\rho$', and the latter is `the sequence of type $\rho^{*}$ with only element $s^{\rho}$'.  The empty sequence for the type $\rho^{*}$ is denoted by `$\langle \rangle_{\rho}$', usually with the typing omitted.  Furthermore, we denote by `$|s|=n$' the length of the finite sequence $s^{\rho^{*}}=\langle s_{0}^{\rho},s_{1}^{\rho},\dots,s_{n-1}^{\rho}\rangle$, where $|\langle\rangle|=0$, i.e.\ the empty sequence has length zero.
For sequences $s^{\rho^{*}}, t^{\rho^{*}}$, we denote by `$s*t$' the concatenation of $s$ and $t$, i.e.\ $(s*t)(i)=s(i)$ for $i<|s|$ and $(s*t)(j)=t(|s|-j)$ for $|s|\leq j< |s|+|t|$. For a sequence $s^{\rho^{*}}$, we define $\overline{s}N:=\langle s(0), s(1), \dots,  s(N)\rangle $ for $N^{0}<|s|$.  
For a sequence $\alpha^{0\di \rho}$, we also write $\overline{\alpha}N=\langle \alpha(0), \alpha(1),\dots, \alpha(N)\rangle$ for \emph{any} $N^{0}$.  By way of shorthand, $q^{\rho}\in Q^{\rho^{*}}$ abbreviates $(\exists i<|Q|)(Q(i)=_{\rho}q)$.  Finally, we shall use $\underline{x}, \underline{y},\underline{t}, \dots$ as short for tuples $x_{0}^{\sigma_{0}}, \dots x_{k}^{\sigma_{k}}$ of possibly different type $\sigma_{i}$.          
\end{nota}    
%
%
\subsection{The classical system $\P$}\label{PIPI}
In this section, we introduce the system $\P$, a conservative extension of $\textsf{E-PA}^{\omega}$ with fragments of Nelson's $\IST$.  

\medskip

To this end, we first introduce the base system $\textsf{E-PA}_{\st}^{\omega*}$.  
We use the same definition as \cite{brie}*{Def.~6.1}, where \textsf{E-PA}$^{\omega*}$ is the definitional extension of \textsf{E-PA}$^{\omega}$ with types for finite sequences as in \cite{brie}*{\S2}.  
The set $\T^{*}$ is defined as the collection of all the terms in the language of $\textsf{E-PA}^{\omega*}$.    
\bdefi\label{debs}
The system $ \textsf{E-PA}^{\omega*}_{\st} $ is defined as $ \textsf{E-PA}^{\omega{*}} + \T^{*}_{\st} + \textsf{IA}^{\st}$, where $\T^{*}_{\st}$
consists of the following axiom schemas.
\begin{enumerate}
\item The schema\footnote{The language of $\textsf{E-PA}_{\st}^{\omega*}$ contains a symbol $\st_{\sigma}$ for each finite type $\sigma$, but the subscript is essentially always omitted.  Hence $\T^{*}_{\st}$ is an \emph{axiom schema} and not an axiom.\label{omit}} $\st(x)\wedge x=y\di\st(y)$,
\item The schema providing for each closed\footnote{A term is called \emph{closed} in \cite{brie} (and in this paper) if all variables are bound via lambda abstraction.  Thus, if $\underline{x}, \underline{y}$ are the only variables occurring in the term $t$, the term $(\lambda \underline{x})(\lambda\underline{y})t(\underline{x}, \underline{y})$ is closed while $(\lambda \underline{x})t(\underline{x}, \underline{y})$ is not.  The second axiom in Definition \ref{debs} thus expresses that $\st_{\tau}\big((\lambda \underline{x})(\lambda\underline{y})t(\underline{x}, \underline{y})\big)$ if $(\lambda \underline{x})(\lambda\underline{y})t(\underline{x}, \underline{y})$ is of type $\tau$.  We usually omit lambda abstraction for brevity.\label{kootsie}} term $t\in \T^{*}$ the axiom $\st(t)$.
\item The schema $\st(f)\wedge \st(x)\di \st(f(x))$.
\end{enumerate}
The external induction axiom \textsf{IA}$^{\st}$ is as follows.  
\be\tag{\textsf{IA}$^{\st}$}
\Phi(0)\wedge(\forall^{\st}n^{0})(\Phi(n) \di\Phi(n+1))\di(\forall^{\st}n^{0})\Phi(n).     
\ee
\edefi
Secondly, we introduce some essential fragments of $\IST$ studied in \cite{brie}.  
\bdefi[External axioms of $\P$]~
\begin{enumerate}
\item$\HAC_{\INT}$: For any internal formula $\varphi$, we have
\be\label{HACINT}
(\forall^{\st}x^{\rho})(\exists^{\st}y^{\tau})\varphi(x, y)\di \big(\exists^{\st}F^{\rho\di \tau^{*}}\big)(\forall^{\st}x^{\rho})(\exists y^{\tau}\in F(x))\varphi(x,y),
\ee
\item $\textsf{I}$: For any internal formula $\varphi$, we have
\[
(\forall^{\st} x^{\sigma^{*}})(\exists y^{\tau} )(\forall z^{\sigma}\in x)\varphi(z,y)\di (\exists y^{\tau})(\forall^{\st} x^{\sigma})\varphi(x,y), 
\]
\item The system $\P$ is $\textsf{E-PA}_{\st}^{\omega*}+\textsf{I}+\HAC_{\INT}$.
\end{enumerate}
\end{defi}
Note that \textsf{I} and $\HAC_{\INT}$ are fragments of Nelson's axioms \emph{Idealisation} and \emph{Standard part}.  
By definition, $F$ in \eqref{HACINT} only provides a \emph{finite sequence} of witnesses to $(\exists^{\st}y)$, explaining its name \emph{Herbrandized Axiom of Choice}.   

\medskip

The system $\P$ is connected to $\textsf{E-PA}^{\omega}$ by the following theorem.    
Here, the superscript `$S_{\st}$' is the syntactic translation defined in \cite{brie}*{Def.\ 7.1}.  
\begin{thm}\label{consresult}
Let $\Phi(\tup a)$ be a formula in the language of \textup{\textsf{E-PA}}$^{\omega*}_{\st}$ and suppose $\Phi(\tup a)^\Sh\equiv\forallst \tup x \, \existsst \tup y \, \varphi(\tup x, \tup y, \tup a)$. If $\Delta_{\intern}$ is a collection of internal formulas and
\be\label{antecedn}
\P + \Delta_{\intern} \vdash \Phi(\tup a), 
\ee
then one can extract from the proof a sequence of closed\footnote{Recall the definition of closed terms from \cite{brie} as sketched in Footnote \ref{kootsie}.\label{kootsie2}} terms $t$ in $\mathcal{T}^{*}$ such that
\be\label{consequalty}
\textup{\textsf{E-PA}}^{\omega*} + \Delta_{\intern} \vdash\  \forall \tup x \, \exists \tup y\in \tup t(\tup x)\ \varphi(\tup x,\tup y, \tup a).
\ee
\end{thm}
\begin{proof}
Immediate by \cite{brie}*{Theorem 7.7}.  
\end{proof}
The proofs of the soundness theorems in \cite{brie}*{\S5-7} provide an algorithm $\mathcal{A}$ to obtain the term $t$ from the theorem.  In particular, these terms 
can be `read off' from the nonstandard proofs.    

\medskip

In light of the results in \cite{sambon}, the following corollary (which is not present in \cite{brie}) is essential to our results.  Indeed, the following corollary expresses that we may obtain effective results as in \eqref{effewachten} from any theorem of Nonstandard Analysis which has the same form as in \eqref{bog}.  It was shown in \cite{sambon, samzoo, samzooII} that the scope of this corollary includes the Big Five systems of Reverse Mathematics and the associated `zoo' (\cite{damirzoo}).  
\begin{cor}\label{consresultcor}
If $\Delta_{\intern}$ is a collection of internal formulas and $\psi$ is internal, and
\be\label{bog}
\P + \Delta_{\intern} \vdash (\forall^{\st}\underline{x})(\exists^{\st}\underline{y})\psi(\underline{x},\underline{y}, \underline{a}), 
\ee
then one can extract from the proof a sequence of closed$^{\ref{kootsie2}}$ terms $t$ in $\mathcal{T}^{*}$ such that
\be\label{effewachten}
\textup{\textsf{E-PA}}^{\omega*} +\QFAC^{1,0}+ \Delta_{\intern} \vdash (\forall \underline{x})(\exists \underline{y}\in t(\underline{x}))\psi(\underline{x},\underline{y},\underline{a}).
\ee
\end{cor}
\begin{proof}
Clearly, if for internal $\psi$ and $\Phi(\underline{a})\equiv (\forall^{\st}\underline{x})(\exists^{\st}\underline{y})\psi(x, y, a)$, we have $[\Phi(\underline{a})]^{S_{\st}}\equiv \Phi(\underline{a})$, then the corollary follows immediately from the theorem.  
A tedious but straightforward verification using the clauses (i)-(v) in \cite{brie}*{Def.\ 7.1} establishes that indeed $\Phi(\underline{a})^{S_{\st}}\equiv \Phi(\underline{a})$.  
\end{proof}
For the rest of this paper, the notion `normal form' shall refer to a formula as in \eqref{bog}, i.e.\ of the form $(\forall^{\st}x)(\exists^{\st}y)\varphi(x,y)$ for $\varphi$ internal.  

\medskip

Finally, the previous theorems do not really depend on the presence of full Peano arithmetic.  
We shall study the following subsystems.   
\bdefi~
\begin{enumerate}
\item Let \textsf{E-PRA}$^{\omega}$ be the system defined in \cite{kohlenbach2}*{\S2} and let \textsf{E-PRA}$^{\omega*}$ 
be its definitional extension with types for finite sequences as in \cite{brie}*{\S2}. 
\item $(\QFAC^{\rho, \tau})$ For every quantifier-free internal formula $\varphi(x,y)$, we have
\be\label{keuze}
(\forall x^{\rho})(\exists y^{\tau})\varphi(x,y) \di (\exists F^{\rho\di \tau})(\forall x^{\rho})\varphi(x,F(x))
\ee
\item The system $\RCAo$ is $\textsf{E-PRA}^{\omega}+\QFAC^{1,0}$.  
\end{enumerate}
\edefi
The system $\RCAo$ is the `base theory of higher-order Reverse Mathematics' as introduced in \cite{kohlenbach2}*{\S2}.  
We permit ourselves a slight abuse of notation by also referring to the system $\textsf{E-PRA}^{\omega*}+\QFAC^{1,0}$ as $\RCAo$.
\begin{cor}\label{consresultcor2}
The previous theorem and corollary go through for $\P$ and $\textsf{\textup{E-PA}}^{\omega*}$ replaced by $\P_{0}\equiv \textsf{\textup{E-PRA}}^{\omega*}+\T_{\st}^{*} +\HAC_{\INT} +\textsf{\textup{I}}+\QFAC^{1,0}$ and $\RCAo$.  
\end{cor}
\begin{proof}
The proof of \cite{brie}*{Theorem 7.7} goes through for any fragment of \textsf{E-PA}$^{\omega{*}}$ which includes \textsf{EFA}, sometimes also called $\textsf{I}\Delta_{0}+\textsf{EXP}$.  
In particular, the exponential function is (all what is) required to `easily' manipulate finite sequences.    
\end{proof}

\subsection{Notations}
We mostly use the notations from \cite{brie}, some of which we repeat.  
\begin{rem}[Notations]\label{notawin}\rm
We write $(\forall^{\st}x^{\tau})\Phi(x^{\tau})$ and $(\exists^{\st}x^{\sigma})\Psi(x^{\sigma})$ as short for 
$(\forall x^{\tau})\big[\st(x^{\tau})\di \Phi(x^{\tau})\big]$ and $(\exists x^{\sigma})\big[\st(x^{\sigma})\wedge \Psi(x^{\sigma})\big]$.     
A formula $A$ is `internal' if it does not involve $\st$; the formula $A^{\st}$ is defined from $A$ by appending `st' to all quantifiers (except bounded number quantifiers).    
\end{rem}
Secondly, we will use the usual notations for rational and real numbers and functions as introduced in \cite{kohlenbach2}*{p.\ 288-289} (and \cite{simpson2}*{I.8.1} for the former).  
\begin{defi}[Real numbers etc.]\label{keepinitreal}\rm
A (standard) real number $x$ is a (standard) fast-converging Cauchy sequence $q_{(\cdot)}^{1}$, i.e.\ $(\forall n^{0}, i^{0})(|q_{n}-q_{n+i})|<_{0} \frac{1}{2^{n}})$.  
We freely make use of Kohlenbach's `hat function' from \cite{kohlenbach2}*{p.\ 289} to guarantee that every sequence $f^{1}$ can be viewed as a real.  We also use the notation $[x](k):=q_{k}$ for the $k$-th approximation of real numbers.    
Two reals $x, y$ represented by $q_{(\cdot)}$ and $r_{(\cdot)}$ are \emph{equal}, denoted $x=_{\R}y$, if $(\forall n)(|q_{n}-r_{n}|\leq \frac{1}{2^{n}})$. Inequality $<_{\R}$ is defined similarly.         
We also write $x\approx y$ if $(\forall^{\st} n)(|q_{n}-r_{n}|\leq \frac{1}{2^{n}})$ and $x\gg y$ if $x>_{\R}y\wedge x\not\approx y$.  Functions $F:\R\di \R$ are represented by $\Phi^{1\di 1}$ such that 
\be\tag{\textsf{RE}}\label{furg}
(\forall x, y)(x=_{\R}y\di \Phi(x)=_{\R}\Phi(y)),
\ee
 i.e.\ equal reals are mapped to equal reals.  Finally, sets are denoted $X^{1}, Y^{1}, Z^{1}, \dots$ and are given by their characteristic functions $f^{1}_{X}$, i.e.\ $(\forall x^{0})[x\in X\asa f_{X}(x)=1]$, where $f_{X}^{1}$ is assumed to be binary.      
\end{defi}
Thirdly, we use the usual extensional notion of equality.  
\begin{rem}[Equality]\label{equ}\rm
Equality between natural numbers `$=_{0}$' is a primitive.  Equality `$=_{\tau}$' for type $\tau$-objects $x,y$ is then defined as follows:
\be\label{aparth}
[x=_{\tau}y] \equiv (\forall z_{1}^{\tau_{1}}\dots z_{k}^{\tau_{k}})[xz_{1}\dots z_{k}=_{0}yz_{1}\dots z_{k}]
\ee
if the type $\tau$ is composed as $\tau\equiv(\tau_{1}\di \dots\di \tau_{k}\di 0)$.
In the spirit of Nonstandard Analysis, we define `approximate equality $\approx_{\tau}$' as follows:
\be\label{aparth2}
[x\approx_{\tau}y] \equiv (\forall^{\st} z_{1}^{\tau_{1}}\dots z_{k}^{\tau_{k}})[xz_{1}\dots z_{k}=_{0}yz_{1}\dots z_{k}]
\ee
with the type $\tau$ as above.  
The system $\P$ includes the \emph{axiom of extensionality}: 
\be\label{EXT}\tag{\textsf{E}}  
(\forall  x^{\rho},y^{\rho}, \varphi^{\rho\di \tau}) \big[x=_{\rho} y \di \varphi(x)=_{\tau}\varphi(y)   \big].
\ee
However, as noted in \cite{brie}*{p.\ 1973}, the so-called axiom of \emph{standard} extensionality \eqref{EXT}$^{\st}$ is problematic and cannot be included in $\P$.  
\end{rem} 
\section{The Loeb measure in $\P$}\label{LEB}
In this section, we discuss the Loeb measure in the context of internal set theory $\IST$, and possible computational aspects thereof.  
This development takes place in the system $\P$ from the previous section.  We assume basic familiarity with \emph{Reverse Mathematics} (RM for short) and we refer to \cites{simpson1, simpson2} for an overview of the latter program, the definition of the `Big Five', and the base theory $\RCA_{0}$ in particular.    

\medskip

First of all, the usual definition of the Loeb measure $L_{M}$ (See Definition \ref{DEP} below) makes use of 
\emph{external} sets, and therefore seems meaningless in $\IST$.  Nonetheless, we shall see that one can give meaning to the formula `$L_{M}(A)=0$' inside $\P$, even though `$L_{M}(A)$' strictly speaking does not exist.  
This is reminiscent of the situation of measure theory in Reverse Mathematics (See e.g.\ \cite{yussie, yuppie}), where the Lebesgue measure is defined as follows in \cite{simpson2}*{X.1.2-3}.  
\bdefi[Lebesgue measure $\lambda$] For $\|g\|:=\int_{0}^{1}g(x)dx$, we define
\[
\lambda(U):=\sup\{\|g\|: g\in C([0,1]) \wedge 0\leq g\leq1 \wedge (\forall x\in [0,1]\setminus U)(g(x)=0) \}.
\]
\edefi
Of course, this supremum does not necessarily exist in weak systems such as the base theory $\RCA_{0}$ of RM, but the formula `$\lambda(U)=_{\R}0$' defined as follows makes perfect sense in weak systems such as $\RCA_{0}$:  
\[
[\lambda(U)=_{\R}0] \equiv (\forall g\in C([0,1])\big[ (0\leq g\leq1 \wedge (\forall x\in [0,1]\setminus U)g(x)=0)\di \|g\|=0\big].
\]
Note that the existence of the Lebesgue measure for open sets is actually equivalent to $\ACA_{0}$ by \cite{simpson2}*{p.\ 391}.  We conclude that while the Lebesgue measure $\lambda$ may not exist in weak systems of RM, the formula $\lambda(U)=_{\R}0$ always is meaningful.  
 
\medskip

Below, we show that a similar trick can used to give meaning to the Loeb measure in $\IST$. 
Thus, fix nonstandard $M$ and consider the grid $\G_{M}=\{\frac{i}{2^{M}} : 0\leq i\leq 2^{M} \}$ on $[0,1]$.
The usual definition of the Loeb measure from \cite{pimpson} is as follows.  
\bdefi[Loeb measure]\label{DEP}
The Loeb measure of a set $B\subseteq\G_{M}$ is $L_{M}(B):= \st(L_{M}^{*}(B))$, where $L_{M}^{*}(B):= \frac{|B|}{2^{M}}$.    
The Loeb measure of a set $A\subseteq [0,1]$ is defined as follows.  
\be\label{cohones}
\st_{M}^{-1}(A):=\{ b\in \G_{M}: (\exists^{\st}a^{1}\in A)(a\approx b)  \}.
\ee
\begin{align*}
C\subset_{\al}D := (\forall E )(E\subset (C\setminus D)\di L_{M}^{*}(E)\approx 0) \qquad (C, D\subseteq \G_{M}).
\end{align*}
\be\label{frinku}
L_{M}(A):=\sup\{L_{M}(B): B\subseteq \G_{M}\wedge B\subset_{\al} \st_{M}^{-1}(A)  \}.
\ee
\be\label{frinku2}
L_{M}^{*}(A):=\sup\{L_{M}^{*}(B): B\subseteq \G_{M}\wedge B\subset_{\al} \st_{M}^{-1}(A)  \}.
\ee
\edefi
Note that $L_{M}(A)=\st(L_{M}^{*}(A))$ since $L_{M}(A)\approx L_{M}^{*}(A)$.  We introduced \eqref{frinku2} as the $\IST$ axiom \emph{Standard Part} is non-constructive, while the standard part map is external, and hence does not exist in $\IST$.  

\medskip

Thirdly, the set $\st_{M}^{-1}(A)$ as in \eqref{cohones} is external, and hence it seems the Loeb measure $L_{M}(A)$ as in \eqref{frinku} cannot be defined in $\IST$.  
Nonetheless, the formula `$L_{M}^{*}(A)\approx 0$' (or equivalently `$L_{M}(A)=0$') does make sense  in $\IST$, as follows:
\begin{align*}
L_{M}^{*}(A)\approx 0
&\equiv (\forall B\subseteq \G_{M})\big[ B\subset_{\al}\st_{M}^{-1}(A) \di    L_{M}^{*}(B)\approx 0\big]\\
&\equiv (\forall B\subseteq \G_{M})\big[ (\forall E)(E \subset (B\setminus \st_{M}^{-1}(A))   \di L_{M}^{*}(E)\approx 0) \di    L_{M}^{*}(B)\approx 0\big]\\
&\equiv (\forall B\subseteq \G_{M})\Big[ \big[(\forall E)\big(  (\forall e\in E)(e\in B\wedge (\forall^{\st}a\in A)(a\not\approx e)   ) \di L_{M}^{*}(E)\approx 0  \big)\big] \di    L_{M}^{*}(B)\approx 0\Big].
\end{align*}
Note that the final formula is a formula of $\IST$ (and can be expressed in far weaker systems such as $\P$).  The formula `$a\in A$' can be replaced by any formula $\Phi(a)$, and we can thus give meaning to the formula $L_{M}^{*}(\{a:\Phi(a)\})\approx 0$.  
We can now say that a property $\Phi$ `holds \emph{almost everywhere} in $[0,1]$' if $L_{M}^{*}(\{a\in [0,1]:\neg\Phi(a)\})\approx 0$.  

\medskip

Fourth, recall that we can obtain computational information from formulas of the form $(\forall^{\st}x)(\exists^{\st}y)\varphi(x,y)$, where $\varphi$ is internal by Corollary \ref{consresultcor}.  
We refer to such formulas as `normal forms'.  
Let us now formulate a normal form for the formula `$L_{M}^{*}(A)\approx 0$'.  A normal form for $(\forall e\in E)(e\in B\wedge (\forall^{\st}a\in A)(a\not\approx e)$ is as follows:
\begin{align}
&~~~(\forall e\in E)(e\in B\wedge (\forall^{\st}a\in A)(a\not\approx e))\notag\\
& \equiv\textstyle(\forall^{\st}a\in A)(\forall e\in E)(\exists^{\st}k^{0})(e\in B\wedge(|a-e|>\frac1k))\notag\\
& \equiv\textstyle(\forall^{\st}a\in A)(\exists^{\st}K^{0^{*}})(\forall e\in E)(\exists k^{0}\in K)(e\in B\wedge(|a-e|>\frac1k))\label{oji}\\
& \equiv\textstyle(\forall^{\st}a\in A)(\exists^{\st}l^{0})\big[(\forall e\in E)(e\in B\wedge(|a-e|>\frac1l))\big],\label{oji2}
\end{align}
where \eqref{oji} follows from applying idealisation \textsf{I}, and \eqref{oji2} follows from defining $l:=\max_{i<|K|}K(i)$ in \eqref{oji}.  
Let $A_{0}(a, E, B, l)$ be the formula in square brackets in \eqref{oji2}, and note that `$L_{M}^{*}(A)\approx 0$' is:
\begin{align*}
&~~~ (\forall B\subseteq \G_{M})\Big[ \big[(\forall E)\big(  (\forall e\in E)(e\in B\wedge (\forall^{\st}a\in A)(a\not\approx e)   ) \di L_{M}^{*}(E)\approx 0  \big)\big] \di    L_{M}^{*}(B)\approx 0\Big].\\
&\equiv\textstyle (\forall B\subseteq \G_{M})\Big[ \big[(\forall E)\big( (\forall^{\st}a\in A)(\exists^{\st}l)A_{0}(a, E, B, l)   \di (\forall^{\st}k)|L_{M}^{*}(E) |\frac{1}k  \big)\big] \di    L_{M}^{*}(B)\approx 0\Big].\\
&\equiv\textstyle (\forall B\subseteq \G_{M})\Big[ \big[(\forall^{\st}k, g)(\forall E)\big( (\forall^{\st}a\in A)A_{0}(a,, E, B, g(a))   \di |L_{M}^{*}(E)|\leq \frac{1}k  \big)\big] \di    L_{M}^{*}(B)\approx 0\Big].\\
&\equiv\textstyle (\forall B\subseteq \G_{M})\Big[ \big[(\forall^{\st}k, g)(\forall E)(\exists^{\st}a\in A)\big( A_{0}(a, E, B, g(a))   \di |L_{M}^{*}(E)|\leq \frac{1}k  \big)\big] \di    L_{M}^{*}(B)\approx 0\Big].\\
&\equiv\textstyle (\forall B\subseteq \G_{M})\Big[ \big[(\forall^{\st}k, g)(\exists^{\st} b^{1^{*}})\underline{(\forall E)(\exists a\in b)\big( a\in A\wedge A_{0}(a, E, B, g(a))   \di |L_{M}^{*}(E)|\leq \frac{1}k}  \big)\big] \di    L_{M}^{*}(B)\approx 0\Big],
\end{align*}
and let $B_{0}(B, k, g , b, A)$ be the underlined formula.  Hence, `$L_{M}^{*}(A)\approx 0$' becomes:
\begin{align*}
&~~~\textstyle (\forall B\subseteq \G_{M})\Big[ \big[(\forall^{\st}k, g)(\exists^{\st} b^{1^{*}})B_{0}(B, k, g, b, A)  \big)\big] \di    L_{M}^{*}(B)\approx 0\Big],\\
&\equiv\textstyle (\forall^{\st}k', h)(\forall B\subseteq \G_{M})\Big[ \big[(\forall^{\st}k, g)B_{0}{(B, k,g,h(k,g), A)}\big] \di    |L_{M}^{*}(B)|\leq \frac{1}{k'}\Big].\\
&\equiv\textstyle (\forall^{\st}k', h)(\forall B\subseteq \G_{M})(\exists^{\st}k, g)\Big[ \big[B_{0}{(B, k,g,h(k,g), A)}\big] \di    |L_{M}^{*}(B)|\leq \frac{1}{k'}\Big].\\
&\equiv\textstyle (\forall^{\st}k', h)(\exists^{\st}w)(\forall B\subseteq \G_{M})(\exists k, g \in w)\Big[ \big[B_{0}{(B, k,g,h(k,g), A)}\big] \di    |L_{M}^{*}(B)|\leq \frac{1}{k'}\Big].\\
\end{align*}
Thus, the final formula is a normal form of `$L^{*}_{M}(A)\approx 0$'.  

\medskip

Fifth, the formula `$L^{*}_{M}(A)\approx 0$' involves a nonstandard number $M$, and one will usually encounter the latter formula somewhere in the scope of the quantifier $(\forall M^{0})(\neg\st(M)\di \dots)$.   Such a nonstandard quantifier place nicely with our normal forms, not just for numbers, but for any finite type $\rho$.
\begin{thm}
For internal $\varphi$, the formula
\be\label{petzi}
(\forall M^{\rho})\big[\neg \st_{\rho}(M)\di (\forall^{\st}x)(\exists^{\st}y)\varphi(x, y, M)\big], 
\ee
is equivalent to a normal form.  
\end{thm}
\begin{proof}
First of all, \eqref{petzi} is equivalent to the following by Definition \ref{debs}:
\[
(\forall M^{\rho})\big[ (\forall^{\st}r^{\rho})(M\ne_{\rho} k)\di (\forall^{\st}x)(\exists^{\st}y)\varphi(x, y, M)\big],
\]
where `$x\ne_{\rho}y$' is an internal formula.
Pushing outside the standard quantifiers as far as possible, we obtain
\[
(\forall^{\st}x)(\forall M^{\rho})(\exists^{\st}r, y)\big[ (M\ne_{\rho} r \di \varphi(x, y, M)\big],
\]
and applying idealisation \textsf{I}, we obtain the following normal form:
\[
(\forall^{\st}x)(\exists^{\st}w)(\forall M^{\rho})(\exists r, y\in w)\big[ (M\ne r \di \varphi(x, y, M)\big].
\]
\end{proof}
By the previous theorem, one can partition a space in infinitesimal pieces $\frac{1}{M}$ for nonstandard $M$, and the associated quantifier.  

\medskip

Sixth, it is a natural question which sets $A$ can be studied using the above normal form for $L^{*}_{M}(A)\approx 0$.
As it turns out, in the definition of $L^{*}_{M}(A)\approx 0$, one can replace `$a\in A$' by any normal form $(\forall^{\st}z)(\exists^{\st}w)\varphi(z, w, a)$, and the resulting 
modification of $L^{*}_{M}(A)\approx 0$ remains a normal form.  In particular, \eqref{oji2} becomes the following normal form with this replacement: 
\begin{align*}
& ~~~\textstyle(\forall^{\st}a\in A)(\exists^{\st}l^{0})\big[(\forall e\in E)(e\in B\wedge(|a-e|>\frac1l))\big]\\
& \equiv\textstyle(\forall^{\st}a)\big[(\forall^{\st}z)(\exists^{\st}a)\varphi(z, w, a)\di   (\exists^{\st}l^{0})\big[(\forall e\in E)(e\in B\wedge(|a-e|>\frac1l))\big]\\
& \equiv\textstyle(\forall^{\st}a, g)\big[(\forall^{\st}z)\varphi(z, g(z), a)\di   (\exists^{\st}l^{0})\big[(\forall e\in E)(e\in B\wedge(|a-e|>\frac1l))\big]\big]\\
& \equiv\textstyle(\forall^{\st}a, g)(\exists^{\st}z, l^{0})\big[\varphi(z, g(z), a)\di   \big[(\forall e\in E)(e\in B\wedge(|a-e|>\frac1l))\big]\big].
\end{align*}
One then obtains a normal form for $L^{*}_{M}(A)\approx 0$ in exactly the same way as for \eqref{oji2}.  
By way of example, we may take `$a\in A$' to be `the function $f$ is nonstandard continuous at $a$', and `$L^{*}_{M}(A)\approx 0$' still has a normal form.    
The same holds for formulas of the form $(\exists^{\st}u)(\forall^{\st}z)(\exists^{\st}w)\varphi(u,z, w, a)$ by the previous, and thus also for \emph{negations} of normal forms.  

\medskip

Seventh, we consider the following alternative definition of the Loeb measure.
\bdefi[Second Loeb measure]\label{DEP2}
The Loeb measure of a set $B\subseteq\G_{M}$ is $L_{M}(B):= \st(L_{M}^{*}(B))$, where $L_{M}^{*}(B):= \frac{|B|}{2^{M}}$.    
The \emph{second} Loeb measure of a set $A\subseteq [0,1]$ is defined as follows.  
\[
\st_{M, 2}^{-1}(A):=\{ b\in \G_{M}: (\exists^{\st}a^{1}, c^{1}\in \R)(a\lessapprox b\lessapprox c)\wedge (\forall^{\st}e^{1}\in \R)(e\in [a, c]\di e\in A))  \}.
\]
\begin{align*}
C\subset_{\al}D := (\forall E )(E\subset (C\setminus D)\di L_{M}^{*}(E)\approx 0) \qquad (C, D\subseteq \G_{M}).
\end{align*}
\be\label{frinku23}
L_{M,2}(A):=\sup\{L_{M}(B): B\subseteq \G_{M}\wedge B\subset_{\al} \st_{M,2}^{-1}(A)  \}.
\ee
\be\label{frinku22}
L_{M,2}^{*}(A):=\sup\{L_{M}^{*}(B): B\subseteq \G_{M}\wedge B\subset_{\al} \st_{M,2}^{-1}(A)  \}.
\ee
\edefi
Note that the formula `$x\not\in \st_{M, 2}^{-1}(A))$' has a normal form similar to that of `$\st_{M}^{-1}(A)$', hence the formula $L^{*}_{M,2}(A)\approx 0$ also has normal form.  

\medskip

In conclusion, we cannot define the Loeb measure $L_{M}(A)$ in $\IST$, but we can give meaning to formula `$L_{M}(A)=0$' (and any other (in)equality in the same way).   
Furthermore, such formulas have normal forms (and therefore carry numerical information), even if we quantify over the nonstandard number $M$.

\section{Bibliography}

\bye